\documentclass[10pt,a4paper]{article}
\usepackage[english]{babel}
\usepackage[T1]{fontenc}
\usepackage{fullpage}
\usepackage{cite}

\usepackage{xargs}
\usepackage{tikz}
\usetikzlibrary{matrix, cd, positioning}

\usepackage{wrapfig}

\usepackage{amsmath}
\usepackage{amssymb}
\usepackage{amsfonts}
\usepackage{amsthm}
\usepackage{mathabx}

\usepackage{multicol}

\usepackage{forloop}

\usepackage{hyperref}

\newcommand{\C}{\mathcal C}

\DeclareMathOperator{\id}{id}

\tikzset{>=Implies}
\newlength{\myline}
\newcommandx*{\doublearrow}[2]{
  \draw[line width=rule_thickness,double equal sign distance,#1] #2;
}
\newcommandx*{\triplearrow}[4][1=0, 2=1]{
  \draw[line width=\myline,double distance=5\myline,#3] #4;
  \draw[line width=\myline,shorten <=#1\myline,shorten >=#2\myline,#3] #4;
}
\newcommandx*{\quadarrow}[4][1=0, 2=2.5]{
  \draw[line width=\myline,double distance=8\myline,#3] #4;
  \draw[line width=\myline,double distance=2\myline,shorten <=#1\myline,shorten >=#2\myline,#3] #4;
}

\usepackage{expl3}

\ExplSyntaxOn
\cs_set_eq:NN \replicate \prg_replicate:nn
\ExplSyntaxOff

\newcommand{\rectangle}[5]{
  \begin{tikzcd}[ampersand replacement=\&, nodes in empty cells, cramped, row sep = #5, column sep = #4]%
  \arrow[d, start anchor = {south}, end anchor = {north}]
  \arrow[\replicate{#2}{d} \replicate{#1}{r}, phantom, start anchor = {real center}, end anchor = {real center}, #3 description]
  \replicate{#1}{ {} \arrow[r, start anchor = {real east}, end anchor = {real west}] \& } 
  \arrow[d, start anchor = {south}, end anchor = {north}] \\
  \replicate{#2-1}{%
  \arrow[d, start anchor = {south}, end anchor = {north}] %
  \replicate{#1}{ {} \& } {} %
  \arrow[d, start anchor = {south}, end anchor = {north}] \\ }
  \replicate{#1}{ {} \arrow[r, start anchor = {real east}, end anchor = {real west}] \& }
  \end{tikzcd}%
}

\newcommand{\smallsquare}[3]{\rectangle{#1}{#2}{#3}{.7cm}{.7cm}}

\author{Maxime LUCAS\footnote{Univ. Paris Diderot, Sorbonne Paris Cit\'e, IRIF, UMR 8243 CNRS, PiR2, INRIA Paris-Rocquencourt, F-75205 Paris, France - \texttt{maxime.lucas@pps.univ-paris-diderot.fr}}}

\title{A cubical Squier's theorem}

\date{}

\begin{document}

\maketitle

\bigskip 

\theoremstyle{plain}
\newtheorem{thm}{Theorem}[section]
\newtheorem{prop}[thm]{Proposition}
\newtheorem{lem}[thm]{Lemma}
\newtheorem{cor}[thm]{Corollary}
\newtheorem{conj}[thm]{Conjecture}
\newtheorem{quest}[thm]{Question}

\theoremstyle{definition}
\newtheorem{defn}[thm]{Definition}
\newtheorem{ex}[thm]{Example}
\newtheorem{remq}[thm]{Remark}
\newtheorem{nota}[thm]{Notation}

Convergent rewriting systems are well-known tools in the study of the word-rewriting problem. In particular, a presentation of a monoid by a finite convergent rewriting system gives an algorithm to decide the word problem for this monoid. In \cite{S87}, \cite{SOK94}, Squier proved that there exists a  finitely presented monoid whose word problem was decidable but which did not admit a finite convergent presentation. To do so, Squier constructed, for any convergent presentation $(G,R)$ of a monoid $M$, a set of syzygies $S$ corresponding to relations between the relations. This construction was later extended (see \cite{GM12}) into the construction of a polygraphic resolution $\Sigma$ of $M$, whose first dimensions coincide with Squier's construction $(G,R,S)$.

However, the construction of the polygraphic resolution has proved to be too complicated to be effectively computed on non-trivial examples. Cubical categories appear to be a promising framework where Squier' theorem and the construction of the polygraphic resolution would be more straightforward. 

This paper is the first step towards this goal. We start by defining the notion of $(2,k)$-cubical categories. We then adapt some classical notions of word rewriting to this cubical setting. Finally, we express and prove a cubical version of Squier's theorem.

\section{Cubical $2$-categories}

\begin{defn}
A cubical $2$-set consists of:
\begin{itemize}
\item Sets $\C_0$, $\C_1$ and $\C_2$, whose objects are respectively called the $0$, $1$ and $2$-cells.
\item Applications $\partial^+,\partial^- : \C_1 \to \C_0$.
\item Applications $\partial_1^+,\partial_1^-,\partial_2^+,\partial_2^-: \C_2 \to \C_1$.
\end{itemize}
Satisfying the following relations for any $\alpha, \beta \in \{+,-\}$: 
\[
\partial^\alpha \partial_2^\beta = \partial^\beta \partial_1^\alpha.
\]
\end{defn}

\begin{nota}
We represent a $1$-cell $f$ in the following way :
\begin{tikzcd}[nodes in empty cells, cramped, row sep = 0.7cm, column sep = 0.7cm]%
 \partial^- f \arrow[r, start anchor = {real east}, end anchor = {real west}, "f"] & \partial^+ f,
\end{tikzcd}
 and a  $2$-cell $A$ as:
\[
\begin{tikzcd}[nodes in empty cells, row sep = 0.7cm, column sep = 0.7cm]%
  \arrow[d, start anchor = {south}, end anchor = {north}, "\partial_2^- A"']
  \arrow[dr, phantom, start anchor = {real center}, end anchor = {real center}, "A" description]
  \arrow[r, start anchor = {real east}, end anchor = {real west}, "\partial_1^- A"] & 
  \arrow[d, start anchor = {south}, end anchor = {north}, "\partial_2^+ A"] \\
  \arrow[r, start anchor = {real east}, end anchor = {real west}, "\partial_1^+ A"' ] &
\end{tikzcd}
\]
\end{nota}

\paragraph*{Cubical $2$-categories.} A cubical $2$-category is a cubical $2$-set equipped with extra structure. See \cite{ABS02} for a formal definition. We give here a run-down of the structure.

\begin{itemize}
\item An operation $\star$ sending any two $1$-cells \begin{tikzcd}[nodes in empty cells, cramped, row sep = 0.7cm, column sep = 0.7cm]%
x
  \arrow[r, "f"] & y
  \arrow[r, "g"] & z
\end{tikzcd} to a $1$-cell \begin{tikzcd}[nodes in empty cells, cramped, row sep = 0.7cm, column sep = 0.7cm]%
  \arrow[r, start anchor = {real east}, end anchor = {real west}, "f \star g"] x & z.
\end{tikzcd}
\item An operation $\epsilon$ sending any $0$-cell $x$ to a $1$-cell \begin{tikzcd}
x \ar[r, "\epsilon x"] &
x,
\end{tikzcd} which we usually represent by \begin{tikzcd}
x \ar[r, equal] &
x.
\end{tikzcd}
\item An operation $\star_1$ (resp. $\star_2$) associating, to any $2$-cells
\smallsquare{1}{1}{"A"} and \smallsquare{1}{1}{"B"}
satisfying $\partial_1^+ A  = \partial_1^- B$ (resp. $\partial_2^+ A  = \partial_2^- B$), $2$-cells 
\[
\rectangle{1}{2}{"A \star_1 B"}{1cm}{.7cm}
\qquad
\smallsquare{2}{1}{"A \star_2 B"}
\]
\item Operations $\epsilon_1, \epsilon_2 : \C_1 \to \C_2$ sending any $1$-cell  \begin{tikzcd}[nodes in empty cells, cramped, row sep = 0.7cm, column sep = 0.7cm]%
  \arrow[r, start anchor = {real east}, end anchor = {real west}, "f"]  & 
\end{tikzcd} to $2$-cells   
\begin{tikzcd}[nodes in empty cells, cramped, row sep = 0.7cm, column sep = 0.7cm]%
  \arrow[d, equal, start anchor = {south}, end anchor = {north}]
  \arrow[dr, phantom, start anchor = {real center}, end anchor = {real center}, "\epsilon_1 f" description]
  \arrow[r, start anchor = {real east}, end anchor = {real west}, "f"] & 
  \arrow[d, equal, start anchor = {south}, end anchor = {north}] \\
  \arrow[r, start anchor = {real east}, end anchor = {real west}, "f"'] &
\end{tikzcd}
and 
\begin{tikzcd}[nodes in empty cells, cramped, row sep = 0.7cm, column sep = 0.7cm]%
  \arrow[d, start anchor = {south}, end anchor = {north}, "f"']
  \arrow[dr, phantom, start anchor = {real center}, end anchor = {real center}, "\epsilon_2 f" description]
  \arrow[r, equal, start anchor = {real east}, end anchor = {real west}] & 
  \arrow[d, start anchor = {south}, end anchor = {north}, "f"] \\
  \arrow[r, equal, start anchor = {real east}, end anchor = {real west}] &
\end{tikzcd}.
\item Operations $\Gamma^-, \Gamma^+ : \C_1 \to \C_2$ sending any $1$-cell \begin{tikzcd}[nodes in empty cells, cramped, row sep = 0.7cm, column sep = 0.7cm]%
  \arrow[r, start anchor = {real east}, end anchor = {real west}, "f"]  & 
\end{tikzcd} to $2$-cells 
\begin{tikzcd}[nodes in empty cells, row sep = 0.7cm, column sep = 0.7cm]%
  \arrow[d, start anchor = {south}, end anchor = {north}, "f"']
  \arrow[dr, phantom, start anchor = {real center}, end anchor = {real center}, "\Gamma^- f" description]
  \arrow[r, "f" , start anchor = {real east}, end anchor = {real west}] & 
  \arrow[d, equal, start anchor = {south}, end anchor = {north}] \\
  \arrow[r, equal, start anchor = {real east}, end anchor = {real west}] &
\end{tikzcd}
and 
\begin{tikzcd}[nodes in empty cells, cramped, row sep = 0.7cm, column sep = 0.7cm]%
  \arrow[d, equal, start anchor = {south}, end anchor = {north}]
  \arrow[dr, phantom, start anchor = {real center}, end anchor = {real center}, "\Gamma^+ f" description]
  \arrow[r, equal, start anchor = {real east}, end anchor = {real west}] & 
  \arrow[d, "f", start anchor = {south}, end anchor = {north}] \\
  \arrow[r, "f"', start anchor = {real east}, end anchor = {real west}] &
\end{tikzcd}.
\end{itemize}
Those operations have to satisfy a number of axioms. In particular, $(\C_0,\C_1,\partial^-, \partial^+, \star, \epsilon)$ and $(\C_1,\C_2,\partial_i^-,\partial_i^+,\star_i,\epsilon_i)$ (for $i  =1,2$) are categories.

\begin{remq}
The cells $\Gamma^\alpha$ and $\epsilon_i$ are completely characterised by their faces. Hence we will omit them when the context is clear in the rest of this paper.
\end{remq}

\paragraph*{Cubical $(2,1)$-categories.}

A cubical $(2,1)$-category is given by a cubical $2$-category $\C$ equipped with an operation $T : \C_2 \to \C_2$ sending any $2$-cell \begin{tikzcd}[nodes in empty cells, row sep = 0.7cm, column sep = 0.7cm]%
  \arrow[d, start anchor = {south}, end anchor = {north}, "\partial_2^- A"']
  \arrow[dr, phantom, start anchor = {real center}, end anchor = {real center}, "A" description]
  \arrow[r, start anchor = {real east}, end anchor = {real west}, "\partial_1^- A"] & 
  \arrow[d, start anchor = {south}, end anchor = {north}, "\partial_2^+ A"] \\
  \arrow[r, start anchor = {real east}, end anchor = {real west}, "\partial_1^+ A"' ] &
\end{tikzcd} to a $2$-cell of shape \begin{tikzcd}[nodes in empty cells, row sep = 0.7cm, column sep = 0.7cm]%
  \arrow[d, start anchor = {south}, end anchor = {north}, "\partial_1^- A"']
  \arrow[dr, phantom, start anchor = {real center}, end anchor = {real center}, "TA" description]
  \arrow[r, start anchor = {real east}, end anchor = {real west}, "\partial_2^- A"] & 
  \arrow[d, start anchor = {south}, end anchor = {north}, "\partial_1^+ A"] \\
  \arrow[r, start anchor = {real east}, end anchor = {real west}, "\partial_2^+ A"' ] &
\end{tikzcd} such that $T^2 = \id_{\C_2}$ and:
\begin{equation}\label{eq:inverse}
\begin{tikzcd}[nodes in empty cells, cramped, row sep = 0.7cm, column sep = 0.7cm]%
\ar[r, equal] 
\ar[d, equal] 
  & 
\ar[r] 
\ar[d] 
\ar[rd, phantom, start anchor = {real center}, end anchor = {real center}, "TA" description] 
  & 
\ar[d] \\
\ar[r] 
\ar[d]  
\ar[rd, phantom, start anchor = {real center}, end anchor = {real center}, "A" description]
  & 
\ar[r] 
\ar[d] 
  & 
\ar[d, equal] 
  \\
\ar[r] 
  & 
\ar[r, equal] 
  &
\end{tikzcd}
=
\begin{tikzcd}[nodes in empty cells, cramped, row sep = 0.7cm, column sep = 0.7cm]%
\ar[r] 
\ar[d] 
  & 
\ar[r, equal] 
\ar[d, equal] 
  & 
\ar[d, equal] \\
\ar[r, equal] 
\ar[d, equal]  
  & 
\ar[r, equal] 
\ar[d, equal] 
  & 
\ar[d] 
  \\
\ar[r, equal] 
  & 
\ar[r] 
  &
\end{tikzcd}
\end{equation}
 
\begin{remq}
The operation $A \mapsto TA$ corresponds to the operation $A \mapsto A^{-1}$ in a globular setting. The equation $T^2 = \id_{\C_2}$ corresponds to the equality $(A^{-1})^{-1}$ and the axiom \eqref{eq:inverse} corresponds to the relation $A \star_1 A^{-1} = 1$.
\end{remq}

\paragraph*{Cubical $2$-groupoid.}
A cubical $2$-groupoid is a cubical $2$-category such that $(\C_0,\C_1)$ is a groupoid (we note \begin{tikzcd}[nodes in empty cells, cramped, row sep = 0.7cm, column sep = 0.7cm]%
  &  \arrow[l, start anchor = {real west}, end anchor = {real east}, "f"'] 
\end{tikzcd} the inverse of a cell \begin{tikzcd}[nodes in empty cells, cramped, row sep = 0.7cm, column sep = 0.7cm]%
  \arrow[r, start anchor = {real east}, end anchor = {real west}, "f"]  & 
\end{tikzcd}) and equipped with operations $S_1, S_2: \C_2 \to \C_2$, sending any $2$-cell \begin{tikzcd}[nodes in empty cells, row sep = 0.7cm, column sep = 0.7cm]%
  \arrow[d, start anchor = {south}, end anchor = {north}, "\partial_2^- A"']
  \arrow[dr, phantom, start anchor = {real center}, end anchor = {real center}, "A" description]
  \arrow[r, start anchor = {real east}, end anchor = {real west}, "\partial_1^- A"] & 
  \arrow[d, start anchor = {south}, end anchor = {north}, "\partial_2^+ A"] \\
  \arrow[r, start anchor = {real east}, end anchor = {real west}, "\partial_1^+ A"' ] &
\end{tikzcd} to $2$-cells of shape:
\[
\begin{tikzcd}[nodes in empty cells, row sep = 1cm, column sep = 1cm]%
  \arrow[dr, phantom, start anchor = {real center}, end anchor = {real center}, "S_1A" description]
  \arrow[r, start anchor = {real east}, end anchor = {real west}, "\partial_1^+ A"] &  \\
  \arrow[u, "\partial_2^- A"]
  \arrow[r, start anchor = {real east}, end anchor = {real west}, "\partial_1^- A"' ] &
  \arrow[u,  "\partial_2^+ A"']
\end{tikzcd}
\qquad
\begin{tikzcd}[nodes in empty cells, row sep = 1cm, column sep = 1cm]%
  \arrow[d,  "\partial_2^+ A"']
  \arrow[dr, phantom, start anchor = {real center}, end anchor = {real center}, "S_2A" description] &  \arrow[l, start anchor = {real west}, end anchor = {real east}, "\partial_1^- A"']
  \arrow[d,  "\partial_2^- A"] \\
  &
  \arrow[l, start anchor = {real west}, end anchor = {real east}, "\partial_1^+ A" ] 
\end{tikzcd}
\] So that $(\C_1,\C_2,\partial_i^-,\partial_i^+,\star_i,\epsilon_i, S_i)$ is a groupoid for $i = 1,2$.

\medskip
Though the proof is not as straightforward as in the globular case,  we still have the following result.
\begin{prop}
A cubical $2$-groupoid is a cubical $(2,1)$-category.
\end{prop}

\section{Word rewriting}

We now apply the structures defined in the previous section to word rewriting.

\begin{defn}
A cubical $(3,2)$-monoid is a cubical $(2,1)$-category object in the category of monoids.
A $(3,1)$-monoid is a cubical $2$-groupoid object in the category of monoids.
\end{defn}

\begin{ex}
In a $(3,2)$-monoid (resp. $(3,1)$-monoid) $\C$, the sets $\C_0,\C_1,\C_2$ are equipped with a structure of monoid, and all the operations on the cells defined previously are morphisms of monoids. In particular, for every $1$-cells $f: u \to u'$ and $g: v \to v'$  in $\C_1$, there is a $1$-cell $fg: uv \to vv'$ in $\C_1$. The fact that $\epsilon_i$ is a morphism of monoids implies that $\epsilon_i (fg) = (\epsilon_i f)(\epsilon_i g)$.
\end{ex}

Polygraphs are presentations for higher-dimensional globular categories and were introduced by Burroni \cite{B93} and by Street under the name of computads \cite{S76} \cite{Street87}. We adapt them here to present cubical $(n,k)$-monoids. 

\begin{defn}
For any set $E$, we denote by $E^*$ the free monoid on $E$. A cubical $2$-polygraph $\Sigma$ is given by two sets $\Sigma_0$, $\Sigma_1$, together with maps $\partial^\alpha: \Sigma_1 \to \Sigma_0^*$ (for $\alpha = \pm$). 

We denote by $\Sigma^*$ (resp. $\Sigma^\top$) the free $2$-monoid (resp. $(2,1)$-monoid) generated by $\Sigma$.
\end{defn}

\begin{defn}
A cubical $(3,2)$-polygraph (resp. $(3,1)$-polygraph) is given by three sets $\Sigma_0$, $\Sigma_1$ and $\Sigma_2$, together with maps $\partial^\alpha:\Sigma_1 \to \Sigma_0^*$ and $\partial_i^\alpha : \Sigma_2 \to \Sigma_1^*$ (resp. $\partial_i^\alpha : \Sigma_2 \to \Sigma_1^\top$).  

We denote by $\Sigma^*$ (resp. $\Sigma^\top$) the free $(3,2)$-monoid (resp. the free $(3,1)$-monoid) generated by $\Sigma$.
\end{defn}

\begin{ex}
If $\Sigma$ is a cubical $(3,2)$-polygraph, the cells of $\Sigma$ and $\Sigma^*$ together with the faces operations can be visualized as follows (a similar diagram could be drawn for $\Sigma^\top$):
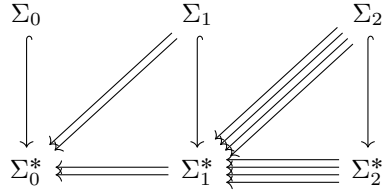
\[
\begin{tikzcd}[sep = 1.5cm]
\Sigma_0 
\ar[d, hookrightarrow]
& 
\Sigma_1 
\ar[d, hookrightarrow]
\ar[ld, shift left = .5mm]
\ar[ld, shift right = .5mm]
& 
\Sigma_2 
\ar[d, hookrightarrow]
\ar[ld, shift left = .5mm]
\ar[ld, shift right = .5mm]
\ar[ld, shift left = 1.5mm]
\ar[ld, shift right = 1.5mm]
\\
\Sigma_0^* 
& 
\Sigma_1^* 
\ar[l, shift left = .5mm]
\ar[l, shift right = .5mm]
& 
\Sigma_2^* 
\ar[l, shift left = .5mm]
\ar[l, shift right = .5mm]
\ar[l, shift left = 1.5mm]
\ar[l, shift right = 1.5mm]
\end{tikzcd}
\]
\end{ex}

The notion of cubical $2$-monoid coincides with the notion of (globular) $2$-monoid as defined by Burroni in \cite{B93}, and a cubical $2$-polygraph is a particular case of (globular) $2$-polygraph. As a consequence, the following classical definitions and results come directly from classical word rewriting theory \cite{GM14}.

\begin{defn}
Let $\Sigma$ be a cubical $2$-polygraph. A \emph{rewriting step} in $\Sigma_1^*$ is a $1$-cell of the form $ufv$, where $f$ is in $\Sigma_1$, and $u$ and $v$ are elements of $\Sigma_0^*$.
\end{defn}

\begin{defn}
Let $\Sigma$ be a cubical $2$-polygraph. It is \emph{confluent} if for any $1$-cells $f$ and $g$ in $\Sigma_1^*$ with the same source, there exist $1$-cells $f'$ and $g'$ in $\Sigma_1^*$ with the same target and such that $\partial^+ f = \partial^- f'$ and $\partial^+ g = \partial^- g'$.

It is \emph{terminating} if there is no infinite sequence of rewriting steps $f_1,\ldots, f_n, \ldots$ satisfying that $\partial^+ f_i = \partial^- f_{i+1}$ for all $i$.

It is \emph{convergent} if it is both terminating and confluent.
\end{defn}

\begin{defn}
Let $\Sigma$ be a cubical $2$-polygraph. A \emph{branching} is a pair of $1$-cells $f,g \in \Sigma_1^*$ with the same source. It is said to be \emph{local} if $f$ and $g$ are rewriting steps.

Up to permutation of $f$ and $g$, there are three distinct types of local branchings:
\begin{itemize}
\item If $f = g$, $(f,g)$ is said to be an \emph{aspherical branching}.
\item If there exists $f',g' \in \Sigma_1^*$ and $u,v \in \Sigma_0^*$ such that $f = f'v$ and $g = ug'$ with $\partial^- f' = u$ and $\partial^- g' = v$, $(f,g)$ is said to be a \emph{Peiffer branching}.
\item Otherwise, $(f, g)$ is said to be an \emph{overlapping branching}.
\end{itemize}

Finally a \emph{critical branching} is a minimal overlapping branching, where overlapping branchings are ordered by the (well-founded) relation: $(f,g) \leq (ufv,ugv)$ for $u,v  \in \Sigma_0^*$
\end{defn}

\section{Squier's theorem}

Before stating Squier's theorem, we need to define the cubical analogue to the notion of globe. 

\begin{defn}
Let $\C$ be a cubical $2$-category.
A \emph{shell} over $\C_1$ is a family of cells $f_i^\alpha$ in $\C_1$, ($i = 1,2$ and $\alpha = +,-$) satisfying $\partial^\alpha f_2^\beta = \partial^\beta f_1^\alpha$ for every $\alpha$ and $\beta$.

A \emph{filler} in $\C_2$ of a shell $S = (f_i^\alpha)$ over $\C_1$ is a cell $A \in \C_2$ satisfying $\partial_i^\alpha A = f_i^\alpha$ for every $i$ and $\alpha$.
\end{defn}

The main result of this paper is the following:

\begin{thm}[Cubical Squier's theorem]\label{thm:Squier_Cubical}
Let $\Sigma$ be a convergent cubical $(3,2)$-polygraph. Suppose that for every critical pair $(f_1,f_2)$ of $\Sigma$, there exists (up to exchange of $f_1$ and $f_2$) a $2$-cell in $\Sigma_2$ whose shell is of the form:
\[
 \begin{tikzcd}[nodes in empty cells, row sep = 0.7cm, column sep = 0.7cm]%
  \arrow[d, start anchor = {south}, end anchor = {north}, "f_2"']
  \arrow[r, start anchor = {real east}, end anchor = {real west}, "f_1"] & 
  \arrow[d, start anchor = {south}, end anchor = {north}] \\
  \arrow[r, start anchor = {real east}, end anchor = {real west}] &
\end{tikzcd}
\]

Then every shell $S$ over $\Sigma_1^\top$ admits a filler in $\Sigma_2^\top$.
\end{thm}

The proof of this result occupies the rest of this article and loosely follows the proof of the globular case from \cite{GM14}.

\begin{lem}
For every local branching $(f_1,f_2)$, there exists a cell $A$ in $\Sigma_2^*$ such that $\partial_1^- A = f_1$ and $\partial_2^- A = f_2$. So $A$ is of the following shape: 
\[
\begin{tikzcd}[nodes in empty cells, row sep = 0.7cm, column sep = 0.7cm]%
\ar[rd, start anchor = {real center}, end anchor = {real center}, phantom, "A" description]
  \arrow[d, start anchor = {south}, end anchor = {north}, "f_2"']
  \arrow[r, start anchor = {real east}, end anchor = {real west}, "f_1"] & 
  \arrow[d, start anchor = {south}, end anchor = {north}] \\
  \arrow[r, start anchor = {real east}, end anchor = {real west}] &
\end{tikzcd}
\]
\end{lem}
\begin{proof}
The proof is similar to the globular case, by distinguishing cases depending on the form of the branching $(f_1,f_2)$.
Note first that if $A$ is a suitable cell for the branching $(f_1,f_2)$, then $TA$ satisfies the conditions for the branching $(f_2,f_1)$, and $uAv$ for the branching $(uf_1v,uf_2v)$. So by hypothesis on $\Sigma_2$, it remains to show that the property holds for aspherical and Peiffer branchings.

If $(f_1,f_2) = (f,f)$ is an aspherical branching, then the $2$-cell \begin{tikzcd}[nodes in empty cells, row sep = 0.7cm, column sep = 0.7cm]%
  \arrow[d, start anchor = {south}, end anchor = {north}, "f"']
  \arrow[dr, phantom, start anchor = {real center}, end anchor = {real center}, "\Gamma^- f" description]
  \arrow[r, "f" , start anchor = {real east}, end anchor = {real west}] & 
  \arrow[d, equal, start anchor = {south}, end anchor = {north}] \\
  \arrow[r, equal, start anchor = {real east}, end anchor = {real west}] &
\end{tikzcd} satisfies the condition.

If $(f_1,f_2) = (fv,ug)$ is a Peiffer branching, then the $2$-cell $(\epsilon_1 f) (\epsilon_2 g)$ satisfies the condition:
\[
\begin{tikzcd}[nodes in empty cells, row sep = 0.7cm, column sep = 0.7cm]%
  \arrow[d, equal, "u"', start anchor = {south}, end anchor = {north}]
  \arrow[dr, phantom, start anchor = {real center}, end anchor = {real center}, "\epsilon_1 f" description]
  \arrow[r, "f" , start anchor = {real east}, end anchor = {real west}] & 
  \arrow[d, equal, "u'", start anchor = {south}, end anchor = {north}] \\
  \arrow[r, "f"', start anchor = {real east}, end anchor = {real west}] &
\end{tikzcd} \cdot \begin{tikzcd}[nodes in empty cells, row sep = 0.7cm, column sep = 0.7cm]%
  \arrow[d, start anchor = {south}, end anchor = {north}, "g"']
  \arrow[dr, phantom, start anchor = {real center}, end anchor = {real center}, "\epsilon_2 g" description]
  \arrow[r, equal, "v", start anchor = {real east}, end anchor = {real west}] & 
  \arrow[d, "g", start anchor = {south}, end anchor = {north}] \\
  \arrow[r, equal, "v'"', start anchor = {real east}, end anchor = {real west}] &
\end{tikzcd} = 
\begin{tikzcd}[nodes in empty cells, row sep = 0.7cm, column sep = 1.5cm]%
  \arrow[d, start anchor = {south}, end anchor = {north}, "ug"']
  \arrow[dr, phantom, start anchor = {real center}, end anchor = {real center}, "(\epsilon_1 f) (\epsilon_2 g)" description]
  \arrow[r, "fv", start anchor = {real east}, end anchor = {real west}] & 
  \arrow[d, "u'g", start anchor = {south}, end anchor = {north}] \\
  \arrow[r, "fv'"', start anchor = {real east}, end anchor = {real west}] &
\end{tikzcd}\] 
\end{proof}

\begin{lem}
For every $f,g \in \Sigma_1^*$ of same source and of target a normal form, the  shell \begin{tikzcd}[nodes in empty cells, row sep = 0.7cm, column sep = 0.7cm]%
  \arrow[d,  start anchor = {south}, end anchor = {north}, "f"']
  \arrow[r, start anchor = {real east}, end anchor = {real west}, "g"] & 
  \arrow[d, equal, start anchor = {south}, end anchor = {north}] \\
  \arrow[r, equal, start anchor = {real east}, end anchor = {real west}] &
\end{tikzcd} admits a filler in $\Sigma_1^*$.
\end{lem}
\begin{proof}

Define the \emph{origin} of a shell $(f_i^\alpha)$ as $\partial^- f_1^- \in \Sigma_0^*$. Let us prove that for any $u \in \Sigma_0^*$, any shell over $\Sigma_1^*$ of origin $u$ and of the form \begin{tikzcd}[nodes in empty cells, row sep = 0.7cm, column sep = 0.7cm]
  \arrow[d,  start anchor = {south}, end anchor = {north}, "f"']
  \arrow[r, start anchor = {real east}, end anchor = {real west}, "g"] & 
  \arrow[d, equal, start anchor = {south}, end anchor = {north}] \\
  \arrow[r, equal, start anchor = {real east}, end anchor = {real west}] &
 \end{tikzcd} admits a filler. We reason by induction on $u$. If $u$ is a normal form, then $f = g = \epsilon u$ and $\epsilon_1 \epsilon u$ is a filler of the shell.

If $u$ is not a normal form, then we can write $f = f_1 \star f_2$ and $g = g_1 \star g_2$ in $\Sigma_1^*$, where $f_1$ and $g_1$ are rewriting steps. Let $A$ be a $2$-cell in $\Sigma_2^*$ such that $\partial_1^- A = f_1$ and $\partial_2^- A = g_1$ (which exists thanks to the previous Lemma). Denote $f' = \partial_1^+ A$ and $g' = \partial_2^+ A$. Then we can apply the induction hypothesis to both $(f',g_2)$ and $(f_2,g')$ defining $2$-cells $B_1$ and $B_2$, and we conclude using the following composite:
\[
\begin{tikzcd}[sep = 1cm]
\ar[r, "f_1"] 
\ar[d, "g_1"']
\ar[rd, phantom, "A" description, start anchor = {south}, end anchor = {north}]
& 
\ar[r, "f_2"] 
\ar[d, "g'" description]
\ar[rd, phantom, "B_2" description, start anchor = {south}, end anchor = {north}]
& 
\ar[d, equal]
\\
\ar[r, "f'" description]
\ar[d, "g_2"']
\ar[rd, phantom, "B_1" description, start anchor = {south}, end anchor = {north}]
& 
\ar[r, equal]
\ar[d, equal]
& 
\ar[d, equal]
\\
\ar[r, equal]
& 
\ar[r, equal]
&
{}
\end{tikzcd}
\]
\end{proof}

\begin{lem}
For every $f \in \Sigma_1^\top$, and every $g_1, g_2 \in \Sigma_1^*$ of target a normal form, the shell 
\begin{tikzcd}[nodes in empty cells, row sep = 0.7cm, column sep = 0.7cm]%
  \arrow[d,  start anchor = {south}, end anchor = {north}, "g_1"']
  \arrow[r, start anchor = {real east}, end anchor = {real west}, "f"] & 
  \arrow[d, start anchor = {south}, end anchor = {north}, "g_2"] \\
  \arrow[r, equal, start anchor = {real east}, end anchor = {real west}] &
\end{tikzcd} admits a filler in $\Sigma_2^\top$.
\end{lem}
\begin{proof}
To prove that the set of $1$-cells $f$ satisfying the Lemma is $\Sigma_1^\top$, we show that it contains $\Sigma_1^*$, and that it is closed under composition and inverses.
\begin{itemize}
\item It contains $\Sigma_1^*$. Indeed, let $f,g_1$ and $g_2$ be $1$-cells in $\Sigma_1^*$. We can form the following composite, where the cell $A$ is obtained by the previous Lemma:
\[
\begin{tikzcd}[nodes in empty cells, row sep = 0.7cm, column sep = 0.7cm]%
\ar[r, "f"]
\ar[d, equal]
& 
\ar[r, equal]
\ar[d, equal]
& 
\ar[d, "g_2"]
\\
\ar[rrd, phantom, "A" description]
\ar[r, "f" description]
\ar[d, "g_1"']
& 
\ar[r, "g_2" description]
& 
\ar[d, equal]
\\
\ar[r, equal]
& 
\ar[r, equal]
& 
\end{tikzcd}
\]
\item It is stable under composition. Indeed, let $f_1, f_2 \in E$ be two composable $1$-cells, and $g_1, g_2 \in \Sigma_1^*$. Let $g_3 \in \Sigma_1^*$ be a $1$-cell such that $\partial^- g_3 = \partial^+ f_1$, and whose target is a normal form. Then the following composite shows that $f_1 \star f_2$ is in $E$, where $A_1$ and $A_2$ exist since $f_1$ and $f_2$ are in $E$:
\[
\begin{tikzcd}[nodes in empty cells]
\ar[rd, phantom, start anchor = {real center}, end anchor = {real center}, "A_1" description]
\ar[r, "f_1", start anchor = {real east}, end anchor = {real west}]
\ar[d, "g_1"']
& 
\ar[rd, phantom, start anchor = {real center}, end anchor = {real center}, "A_2" description]
\ar[r, "f_2", start anchor = {real east}, end anchor = {real west}]
\ar[d, "g_3" description]
& 
\ar[d, "g_2"]
\\
\ar[r, equal, start anchor = {real east}, end anchor = {real west}]
& 
\ar[r, equal, start anchor = {real east}, end anchor = {real west}]
&  
\end{tikzcd}
\]
\item It is stable under inverses. Indeed, let $f \in E$, and let $g_1, g_2 \in \Sigma_1^*$. We can construct the following cell, where $A$ comes from the fact that $f$ is in $E$, applied to the pair $(g_2,g_1)$: 
\begin{tikzcd}[nodes in empty cells]%
  \arrow[d, "g_1"', start anchor = {south}, end anchor = {north}]
  \arrow[dr, phantom, start anchor = {real center}, end anchor = {real center}, "S_2B" description]
 & 
  \arrow[l, "f"' , start anchor = {real west}, end anchor = {real east}]
  \arrow[d, "g_2", start anchor = {south}, end anchor = {north}] \\
  \arrow[r, equal, start anchor = {real east}, end anchor = {real west}] &
\end{tikzcd}
\end{itemize}
\end{proof}

\begin{proof}[Proof of Theorem \ref{thm:Squier_Cubical}]
Let us fix a shell $(f_i^\alpha)$ over $\Sigma_1^\top$. The following cell is a filler of $f_i^\alpha$. The $1$-cells $g_1$, $g_2$, $g_3$ and $g_4$ are arbitrary $1$-cells in $\Sigma_1^*$, with the appropriate source the normal form as target. The cells $B_1,B_2,B_3$ and $B_4$ are obtained by the previous Lemma and rotated as needed using $T$, $S_1$ and $S_2$.

\[
\begin{tikzcd}[nodes in empty cells, row sep = 0.9cm, column sep = 0.9cm]%
\ar[r, equal]
\ar[d, equal]
& 
\ar[r, "f_1^-"]
\ar[d, "g_1" description]
\ar[rd, phantom, start anchor = {real center}, end anchor = {real center}, "B_1" description]
& 
\ar[r, equal]
\ar[d, "g_2" description]
& 
\ar[d, equal]
\\
\ar[r, "g_1" description]
\ar[d, "f_2^-"']
\ar[rd, phantom, start anchor = {real center}, end anchor = {real center}, "B_2" description]
& 
\ar[r, equal]
\ar[d, equal]
&
\ar[rd, phantom, start anchor = {real center}, end anchor = {real center}, "B_3" description]
\ar[d, equal] 
&
\ar[l, "g_2" description] 
\ar[d, "f_2^+"]
\\
\ar[r, "g_3" description]
\ar[d, equal]
& 
\ar[r, equal]
\ar[rd, phantom, start anchor = {real center}, end anchor = {real center}, "B_4" description]
& 
& 
\ar[d, equal]
\ar[l, "g_4" description]
\\
\ar[r, equal]
& 
\ar[r, "f_1^+"']
\ar[u, "g_3" description]
& 
\ar[r, equal]
\ar[u, "g_4" description]
& 
\end{tikzcd}
\]
\end{proof}
\bibliography{cst}
\bibliographystyle{plain}
\end{document}